\documentclass[12pt, 14paper,reqno]{amsart}
\usepackage{enumitem}

\setlist[enumerate]{font={\bfseries}}
\setlist[enumerate,1]{label={(\arabic*)}}

\usepackage{amsmath,amsfonts,amssymb}
\usepackage{mathrsfs}
\usepackage{longtable}
\usepackage[breaklinks]{hyperref}
\usepackage{graphicx}
\usepackage{mathtools}
\usepackage{multirow}
\usepackage{multicol}
\usepackage{array}

\theoremstyle{plain}

\newtheorem{thm}{Theorem}[section]

\newtheorem{cor}{Corollary}[section]

\newtheorem{defi}{Definition}[section]

\newtheorem{thma}{Theorem}

\numberwithin{equation}{section}

\begin{document} 
\title{A Matrix Analogue of Schur-Siegel-Smyth Trace Problem}
\author{Srijonee Shabnam Chaudhury}
\address{ Srijonee Shabnam Chaudhury
@Harish-Chandra Research Institute, A CI of Homi Bhabha National Institute,
Chhatnag Road, Jhunsi,  Allahabad 211 019, India.}
\email{srijoneeshabnam@hri.res.in}
\maketitle
\begin{abstract}

Let $\mathcal{S}$ be the set of all  positive-definite, symmetrizable integer matrices with non-zero upper and lower diagonal and $\mathcal{T}$ to be the set of all positive-definite real symmetric matrices with nonzero upper diagonal such that all non-zero entries are square-roots of some positive integers and the matrices satisfy a certain cycle condition.

In this paper, for any $n \times n$ matrix $A \in \mathcal{S} \cup \mathcal{T}$ and any $k \in \mathbb{N}$ we find a general lower bound for $Tr_{2^k}(A)$, i.e, the sum of $2^k$-th power of eigenvalues of $A$, which depends on $n$ as well as some other variables. In particular, we obtain the best possible lower bound for $Tr_2(A) $ that is $6n - 5$. As a strong outcome of this result we show that the smallest limit point of $\overline{Tr_2(A)} = \frac{Tr_2(A)}{n}$ is $6$. This is a solution of an analogue of ``Schur - Siegel - Smyth trace problem" for characteristic polynomials of matrices in $\mathcal{S} \cup \mathcal{T}$. We also obtain a lower bound of smallest limit point of $\overline{Tr_{2^k}(A)}$ for any positive integer $k > 1$ and for the same set of matrices. Furthermore, we exhibit that the famous results of Smyth on density of absolute trace measure and absolute trace-2 measure of totally positive integers are also true for the set of symmetric integer connected positive definite matrices.
\end{abstract}
\maketitle

\section{Introduction:}
Let $\alpha$ be a totally positive algebraic integer. For any positive integer $k$ the sum of $k$-th power of all conjugates of $\alpha$ is called the \textit{trace-$k$ measure} of $\alpha$ and it is denoted by $Tr_k(\alpha)$. If we divide this measure by the degree of $\alpha$ then we get the \textit{absolute trace-$k$ measure} of $\alpha$ which is denoted by $\overline{Tr_k(\alpha)}$. For $k = 1$ the \textit{trace-$1$ measure of $\alpha$} and \textit{absolute trace-$1$ measure} of $\alpha$ are simply called the \textit{trace} of $\alpha$ and $\textit{absolute trace}$ of $\alpha$ respectively.

The ``Schur - Siegel - Smyth trace problem" states that --

\medskip

\textit{ For any $\rho<2$ prove that there exists only finitely many totally positive algebraic integers with $\overline{Tr(\alpha)} < \rho$ and if possible find all such algebraic integers.}

This problem is still open but some remarkable progress has been made in recent times. See \cite{S}, \cite{AP1},\cite{Si}, \cite{ABP}, \cite{AP2}, \cite{DW}, \cite{F1} \cite{S2} for history and current development on this problem.

We would like to mention one of the most important results on this problem here which helped to generalise the spectrum of this problem. For this let us define \textit{Spectrum of Absolute Trace-$k$ measures} that is, set of Absolute Trace-$k$ measure for all totally positive algebraic integer $\alpha$ except $0$ and $\pm 1$. We denote it by 
    $$
    \mathcal{A}_k(\alpha) = \{  \overline{{{Tr_{k}}}}(\alpha) |\alpha \neq 0, \pm 1 \ \text{is a totally positive algebraic integer} \}
    $$
    In 1984, Smyth \cite{S1} carried out a detailed analysis of the structure of $\mathcal{A}_k(\alpha)$ for $k \geq 1$ and showed that for $k=1$ and $\nu_1 = 1.7719$ there exists only finitely many elements of $\mathcal{A}_1$ in $(1, \nu_1)$, the number of elements is unknown in $(\nu_1, 2)$ and the set is dense in $[2, \infty)$. Similarly for $k = 2$ and $\nu_2= 5.19610$, he showed that there exists finitely many elements of $\mathcal{A}_2$ in $(1, \nu_2)$, the number of elements is unknown in $(\nu_2, 6)$ and the set is dense in $[6, \infty)$. For higher values of $k$ he also obtained similar pattern.

This result gave rise to the analogue of ``Schur - Siegel - Smyth trace problem". In particular, for $k=2$ the statement is the following --

\textit{ For any $\nu < 6$ prove that there exists only finitely many totally positive algebraic integers with $\overline{Tr(\alpha)} < \nu$ and if possible find all such algebraic integers.}

For more details on this problem see for example \cite{F2}, \cite{LW}, \cite{CXQ}.

Although the ``Schur - Siegel - Smyth trace problem" and the analogous problems are strenuous to handle with, from some recent results [See \cite{SM}, \cite{MY}] it seems that it is not very demanding to prove some results on the same if we try to find it for some combinatorial objects (such as graph, matrix, etc.) attached with an algebraic integer.

In 2013, Mckee and Yatsyna \cite{MY} obtained that $Tr(A) \geq 2n-1$ for an $n \times n$ connected, positive-definite integer symmetric matrix $A$. Using this bound they solved the matrix analogue of ``Schur - Siegel - Smyth trace problem" for this class of matrices. After that in 2020,  Smyth and Mckee \cite{SM} extended this result for a larger class  which consists of all connected, positive-definite, symmetrizable integer matrices and showed that in this case also  the same $2n-1$  serves as the lower bound of trace of matrices.

As a continuation of the study done in last two papers, in this paper we obtain a best-possible lower bound of \textit{Trace-$2$ measure} of any positive-definite symmetrizable integer matrix with non-zero upper and lower diagonal and any positive-definite symmetric matrix with non-zero upper diagonal and with some other certain conditions. And as a generalisation of this result we also obtain a lower bound for \textit{Trace-$2^k$ measure} for the same set of matrices.

Precisely, our main result proves that --

\begin{thma} \label{th1}
    Let $A$ be a positive-definite symmetrizable integer matrix with non-zero upper and lower diagonal or a positive-definite symmetric matrix with non-zero upper diagonal such that all non-zero entries are square root of some positive integers and the matrix satisfy a certain cycle condition. Then

    \begin{itemize}

        \item $Tr_2 (A) \geq 6n-5 $.

        \item $Tr_{2^k} (A) \geq 2^{2^{k-1}} + (n-2)6^{2^{k-1}} +  5^{2^{k-1}}$ for all $k \in \mathbb{N}$. 

    \end{itemize}
\end{thma}

From the fact that for any $n\in \mathbb{N}$ there exist matrix $A$ of order $n$ with $Tr_2 (A) = 6n-5 $  we solve the analogue of 'Schur - Siegel - Smyth trace problem' for \textit{Trace $2$ measure}
of matrices defined in \ref{th1} and proved the following --

\begin{thma}
    For any matrix $A$ which is one of the above mentioned two types (mentioned in \ref{th1} )

\begin{itemize}

    \item The smallest limit point of $\overline{Tr_{2}(A)}$ is $6$.

    \item The smallest limit point of $\overline{Tr_{2^k}(A)}$ is greater than or equals to $5^{2^{k-1}}$.
    
\end{itemize}
    
\end{thma}

On the other hand, as matrix analogues of Smyth's results on density of absolute trace and absolute trace-2 measure of totally positive algebraic integers we prove that --

\begin{thma}
    The set of --
    \begin{itemize}
        \item absolute trace of connected, positive-definite, symmetric integer matrices is dense in $[2, \infty)$.

        \item absolute trace-2 measure of connected, positive-definite, symmetric integer matrices is dense in $[6, \infty)$.
    \end{itemize}
    
\end{thma}

Before going to the proofs, in next section we define some concepts which are useful for this paper.

\section{Some Basics:} \label{S}

A matrix is a way of representing a morphism in the category of finite dimensional vector spaces. Here we give definitions and basic properties of some special types of matrices and the graphs and eigenvalues associated with them.

Let $A \in M_n(\mathbb{Z})$, that is, $A$ be an $n \times n$ integer matrix. For any matrix $A$ (square or not) we define the transpose of $A$ by $A^T$. 

\begin{defi}: A real square matrix $A = (a_{ij})$ of order $n > 2$ is called \textbf{sign symmetric} if

$$
a_{ij} = a_{ji} = 0 \hspace{5mm} \text{or} \hspace{5mm} a_{ij}a_{ji} > 0, \hspace{2mm} i \neq j, \hspace{5mm} i,j \in \{ 1,...n \}
$$. 

We denote 
$$\mathcal{S}_{n}^{\pm} \coloneqq \text{the set of all sign symmetric matrices of order n} 
$$
And
$$
\mathcal{S}^{\pm} \coloneqq \bigcup _{n \in \mathbb{N}} \mathcal{S}_{n}^{\pm}
$$

\end{defi}

\begin{defi}: An $n \times n$ real matrix $A = (a_{ij})$ is said to satisfy \textbf{cycle condition} if for any finite sequence $i_1, i_2,..., i_k \in \{ 1,...,n \}$ where $3 \leq k \leq n $
$$
a_{i_1i_2}a_{i_2i_3}...a_{i_ki_1} = a_{i_2i_1}a_{i_3i_2}...a_{i_1i_k}
$$
\end{defi}

\begin{defi}: An $n \times n$ real matrix $A$ is called \textbf{symmetrizable} if $A$ is sign symmetric and satisfies the cycle condition. 

Or equivalently,

if there exists a diagonal matrix $D$ with all diagonal entries are positive, such that $D^{-1}AD$ is a real symmetric matrix.
\end{defi}
\begin{defi}: An $n \times n$ real matrix $S = (s_{ij})$ is said to satisfy \textbf{rational cycle condition} if for any finite sequence  $i_1, i_2,..., i_k \in \{ 1,...,n \}$ where $3 \leq k \leq n $
$$ 
s_{i_1i_2}s_{i_2i_3}...s_{i_ki_1} \in \mathbb{Q}
$$
\end{defi}
Let us also define 

$\mathcal{T}_{n}^c \coloneqq $ Set of all  $n \times n$   real symmetric matrices with entries in $\sqrt{N_{0}}$  and satisfies rational cycle condition

Then,
$$
\mathcal{T}^{c} = \bigcup_{n \in \mathbb{N}} \mathcal{T}_{n}^{c}
$$

Where, $\sqrt{\mathbb{N}_0} \coloneqq \pm \sqrt{\mathbb{N}} \cup \{0\} = \{ \pm \sqrt{a}, 0 | a\in \mathbb{N} \} $

\begin{defi}\label{d25}
For any $A = (a_{ij}) \in \mathcal{S}_{n}^{\pm}$  a map 
$$
\phi : \mathcal{S}_{n}^{\pm} \rightarrow \mathcal{T}_{n}^c
$$
is defined by $\phi(a_{ij}) = sgn(a_{ij}) \sqrt{a_{ij}a_{ji}}$

If $A$ is a symmetrizable matrix then $\phi$ is called a \textbf{symmetrization map}.
\end{defi}

\begin{defi}: Let $A = (a_{ij})$ be an $n \times n$ real symetrizable  matrix. A \textbf{directed graph $\mathbf{G_A}$ with $\mathbf{n}$ vertices associated with the matrix $\mathbf{A}$} can be defined by the following way--
 \begin{itemize}
     \item \textit{vertex set}: Consider each column of the matrix $A$ as a vertex and give a label of each vertex from  $1$ to $n$ such that the $i$-th vertex corresponds to the $i$-th column of the matrix for all $i \in \{ 1,2 ,..., n \}$.  
    
    \item \textit{edge set}:  There exists a directed edge from vertex $i$ to $j$ if $a_{ij} \neq 0$. Thus, a set of all non-zero entries of matrix $A$ reflect an edge of the associated directed graph.
    
\end{itemize}   
    \end{defi}
    
    Here is an example.
  $$ A=   
    \begin{pmatrix}
2 & 16 & 12\\
1 & 6 & 4\\
3 & 16 & 10 \\
\end{pmatrix}
$$

is a $3 \times 3$ symetrizable integer matrix. The directed graph $G_A$ associated with $A$ is generated by the points  \{ $(2,16,12)$, $(1,6,4)$, $(3,16,10)$ \} where we denote $(2,16,12)$, $(1,6,4)$ and $(3,16,10)$ as  $1$-st, $2$-nd and $3$-rd vertex respectively.

\begin{defi}:  A matrix is said to be \textbf{connected} if any matrix $B$ congruent to $A$ is not in block diagonal form. If a matrix is connected its associated graph is also connected. 
\end{defi}

\textit{Example:} The matrix $A$ given above is an example of a connected matrix.

Now we will define the totally positive eigenvalues of a square integer matrix $A$. 

\begin{defi}: Let $P_A =$ Determinant of $(xI - A) \in \mathbb{Z}[X]$ be the characteristic polynomial of $A$. Then an eigenvalue of $A$, that is, a root of $P_A$, is called \textbf{totally positive} if all of its conjugates (which are also eigenvalues of $A$) are positive real number. 
    
\end{defi}

A matrix is said to be \textbf{positive definite} if and only if all of its eigenvalues are totally positive. 

Let us define some special types of connected, positive definite matrices here.

$$
\mathcal{S}_n = \{ A \in M_n(\mathbb{Z})| A  \text{ is symmetrizable, positive-definite with non-zero upper and lower diagonal } \} 
$$

And 

$$
\mathcal{S} \coloneqq \bigcup_{n \in \mathbb{N}} \mathcal{S}_n
$$
 
$$
\mathcal{T}_n \coloneqq \{ A \in M_n(\mathbb{R}) |  A   \text{ is positive-definite  with non-zero upper diagonal and } A \in \mathcal{T}_{n}^c \} 
$$

And,

$$
\mathcal{T} = \bigcup_{n \in \mathbb{N}} \mathcal{T}_n
$$

Now, we first define some quantities related to a totally positive algebraic integer and then analogously interpret them in matrix set-up.

\begin{defi}: Let $\alpha$ be a totally positive algebraic integer of degree $d \geq 2$ with conjugates $\alpha= \alpha_1$, $\alpha_2$,..., $\alpha_d$. For any $k \in \mathbb{N}$ we define the \textbf{Trace-$k$ measure} of $\alpha$ by
    
    $$
    {{Tr_{k}}}(\alpha) = \sum_{i=1}^{d} \alpha_{i}^{k}
    $$

\end{defi}    
\begin{defi}:    
    The arithmetic mean of $k$-th power of all conjugates of $\alpha$ is called the \textbf{absolute Trace-$k$ measure}. We denote it by

    $$
   \overline{{{Tr_{k}}}}(\alpha) = \frac{\sum_{i=1}^{d} \alpha_{i}^{k}}{d}
    $$
    \end{defi}

We can use these above mentioned quantities for totally positive algebraic integers in exactly similar way for any real square matrix $A$.

Let $\beta_1$,..., $\beta_n$ be set of all eigenvalues of the matrix $A$.

Then we define --

\begin{itemize}
    \item \textit{Trace-$k$ measure of $A$}  $\coloneqq 
    {{Tr_{k}}}(A) \coloneqq \sum_{i=1}^{n} \beta_{i}^{k}$
    
    \item  \textit{Absolute Trace-$k$ measure of $A$} $\coloneqq \overline{{{Tr_{k}}}}(A) \coloneqq \frac{\sum_{i=1}^{n} \beta_{i}^{k}}{n}$

\end{itemize}

\section{Statement and Proof of Theorem \ref{1} And Its Corollaries}

\begin{thm}\label{1}
For any matrix $A \in \mathcal{T}_n \cup \mathcal{S}_n$,
$$
Tr_{2^k} (A) \geq 2^{2^{k-1}} + (n-2)6^{2^{k-1}} +  5^{2^{k-1}}
$$
for any $k \in \mathbb{N}$. 

In particular, $Tr_2 (A) \geq 6n-5 $.
\end{thm}
\begin{proof}: We first prove the inequality on the lower bound of trace $2^k$-measure for matrices in $\mathcal{T}_n$

Case I: Let the multi-set of eigenvalues of $A\in \mathcal{T}_n$ is $\{\lambda_1,...,\lambda_n\}$

First, we show that for any square matrix $A$ with the multi-set of eigenvalues $\{\lambda_1,...,\lambda_n\}$ 

$$
Tr(A^{2^k}) = \sum_{i=1}^n \lambda_i^{2^k}
$$
From Schur Triangularization \cite{Sch} theorem

$$
A = UTU^{-1} 
$$
For some triangular matrix $T$ and unitary matrix $U$. 

Then
$$
A^{2^k} = UT^{2^k}U^{-1}
$$

Since the eigenvalues of a triangular matrix are diagonal elements repeated according to multiplicity and the multi-set of eigenvalues is preserved under similarity therefore
\begin{equation}\label{eq41}
Tr(A^{2^k}) = Tr(T^{2^k})  = \sum_{i=1}^n \lambda_i^{2^k}
\end{equation}
As the diagonal matrix corresponding to $T$ is $diag(T) = (\lambda_1 , \lambda_2,  ... \lambda_n) \implies diag(T)^{2^k} = (\lambda_1^{2^k},  \lambda_2^{2^k},  ... \lambda_n^{2^k})$.

Now we show that 

\begin{equation}
Tr(A^{2^k}) \geq \sum_{i=1}^n (\sum_{j=1}^n a_{ij}^2)^{2^{k-1}}
\end{equation}
For a square symmetric real matrix $A$ and any positive integer $k$.

For $k=1$ we show that 
$$ 
(A^2)_{ii} = \sum _{j=1}^n a_{ij}^2  \hspace{5mm} \text{where} \hspace{2mm}i = 1,...,n
$$ 
and 
$$
Tr(A^2) \geq \sum_{i=1}^n \left(\sum_{j=1}^n a_{ij}^2\right)
$$
Actually for $k=1$ the equality holds.

For any square symmetric matrix $A = (a_{ij})$ where $i,j \in \{1,...n\}$ 
$$
A^2 = \left(\sum_{k=1}^n a_{ik} a_{kj}\right)_{ij}
$$
Then the diagonal elements of $A^2 = \sum_{i,j=1}^n a_{ij}a_{ji} = \sum_{i,j=1}^n a_{ij}^2 $ (As $A$ is symmetric matrix) and $Tr(A^2) = \sum_{i=1}^n \sum_{j=1}^n a_{ij}^2$

For $k=2$ we now show that 

$$ 
(A^4)_{ii} = \left(\sum _{j=1}^n a_{ij}^2 \right)^2  \hspace{5mm} \text{where} \hspace{2mm}i = 1,...,n
$$ 
And 
$$
Tr(A^4) \geq \sum_{i=1}^n \left(\sum_{j=1}^n a_{ij}^2\right)^2
$$
We know that, $ (A^2)_{ii} = \sum _{j=1}^n a_{ij}^2 = a'_{ii} \hspace{2mm} (\text{say})  \hspace{5mm} \text{where} \hspace{2mm}i = 1,...,n $

Then $ (A^4)_{ii} = \sum _{j=1}^n (a'_{ij})^2 = (a'_{ii})^2 + \sum_{i,j =1 i\neq j}^n (a'_{ij})^{2} \geq (a'_{ii})^2 = (\sum _{j=1}^n a_{ij}^2)^2 $

Let us assume that the result is true for $(A^{2^{k-1}})$ for any square symmetric real matrix $A$ and any positive integer $k > 1$.

Take $B = A^{2^{k-1}}$ then $Tr(A^{2^k})= Tr(B^2) = \sum_{i=1}^n b_{ii}^2$ where $B = (b_{ij})_{1 \leq i,j \leq n}$.

By induction hypothesis, 
$$
b_{ii} \geq \left(\sum _{j=1}^n a_{ij}^2\right)^{2^{k-2}}  
$$
This implies,  
$$
(b_{ii})^2 \geq \left(\sum _{j=1}^n a_{ij}^2\right)^{2^{k-1}}  
$$
Therefore, $ Tr(A^{2^k}) = Tr(B^2) \geq \sum_{i=1}^n \left(\sum_{j=1}^n a_{ij}^2\right)^{2^{k-1}}$. Then from equation \eqref{eq41} we can conclude that, 
$$
Tr_{2^k}(A) \geq \sum_{i=1}^n \left(\sum_{j=1}^n a_{ij}^2\right)^{2^{k-1}}
$$

Let us choose $A  = (a_{ij})_{1 \leq i,j \leq n}\in \mathcal{T}_n$ such that the graph $T_A$ of $A$ is a tree, $Tr(A) = 2n - 1$ and upper diagonal entries of $A$ are $1$ or $-1$. Since $Tr(A)= 2n - 1$ and all diagonals are positive therefore atleast one diagonal entry is $1$. Without loss of generality we assume that $ a_{11}= 1$. From Proposition $18$ of \cite{SM} we obtain that for any matrix $B = (b_{ij}) \in \mathcal{T}_n$, $Tr(B) \geq 2n-1$ and $b^2_{i i+1} = b^2_{i+1 i} \geq 1 $ for $i = 1,..,n-1$ therefore for estimating a lower bound of \textit{Trace $2^k$ measure} of matrices in $\mathcal{T}_n$ we choose $A\in \mathcal{T}_n$ satisfying these conditions.

Let $N_1$, $N_2$,..., $N_m$ are precisely all diagonal entries of $A$ which are greater than $2$. We can write $N_i = 2 + l_i$ for $i = 1,..,m$. Then we will show that exactly $l_1 + l_2 +...+ l_m +1$ number of diagonal entries of $A$ are $1$ and $n - (m + l_1 + l_2 +...+ l_m +1)$ number of diagonal entries of $A$ are $2$.

Let $l$ number of diagonal entries be $1$. Since there are precisely $m$ number of diagonal entries greater than $2$ therefore we can say exactly $n - (l+m)$ number of diagonal entries are $2$. Then
$$
Tr(A) = 2n - 1 \implies \sum_{i=1}^m N_i + l + 2 \{n - (l + m) \} = 2n - 1 \implies l = \sum_{i=1}^m l_i + 1
$$

Let us denote $l_1 + l_2 +...+ l_m \coloneqq M$

Since $Tr(A)= 2n - 1$ and all diagonals are positive therefore atleast one diagonal entry is $1$.

Then we get three cases --

Case I(a): Fix $a_{nn} = 2$

Since there are exactly $M+1$ number of rows with diagonal entries $1$ and first row is one of them therefore for these rows we have 
\begin{equation}\label{eq11}\
\begin{split}
&\left(\sum_{j=1}^n a_{1 j}^2\right)^{2^{k-1}} + M \left(\sum_{j=1}^n a_{i j}^2 \right)^{2^{k-1}} \\
&=  \left(\sum_{j=1}^n a_{1 1}^2+ a_{1 2}^2 \right)^{2^{k-1}} + M  \left(\sum_{j=1}^n a_{i i}^2+ a_{i i+1}^2 + a_{i i-1}^2\right)^{2^{k-1}}
= 2^{2^{k-1}}  + M 3^{2^{k-1}}
\end{split}
\end{equation}
here $i \in \{2,..n -1 \} $ such that $a_{ii} = 1$.

There are exactly $m$ number of rows with diagonal entries greater than $2$. For these rows we have
\begin{equation}\label{eq12}
\left(\sum_{i}\sum_{j=1}^n a_{i j}^2\right)^{2^{k-1}} = \left(\sum_{i}\sum_{j=1}^n a_{i i}^2+ a_{i i+1}^2 + a_{i i-1}^2\right)^{2^{k-1}} = \sum_{k=1}^m \left( N_{k}^2 + 2 \right)
\end{equation}
here $i \in \{2,..n -1 \} $ such that $a_{ii} > 2$.

Also, there exists exactly $\{n - (m+M+1)\}$ number of rows with diagonal entries  $2$ and $n$-th row is one of them. For these rows we have
\begin{equation}\label{eq13}
\begin{split}
&\left( n - (m+M+1) \right) \left(\sum_{j=1}^n a{i j}^2 \right)^{2^{k-1}} + \left(\sum_{j=1}^n a_{n j}^2\right)^{2^{k-1}}\\
&= \left( n - (m+M+1) \right) \left(\sum_{j=1}^n a_{i i}^2+ a_{i i+1}^2 + a_{i i-1}^2\right)^{2^{k-1}} + \left(\sum_{j=1}^n a_{n n}^2+ a_{n-1 n}^2 \right)^{2^{k-1}} \\
&= \left( n - (m+M+1) \right) 6^{2^{k-1}}+ 5^{2^{k-1}}
\end{split}
\end{equation}
here $i \in \{2,..n -1 \} $ such that $a_{ii} = 2$.

Finally, adding equation \ref{eq11},\ref{eq12},\ref{eq13} we get --
\begin{equation}
\begin{split}
    Tr_{2^k}(A) &=  2^{2^{k-1}}  + 3^{2^{k-1}} M + \sum_{k=1}^m \left( N_{k}^2 + 2 \right)^{2^{k-1}} + \left( n - (m+M+1) \right) 6^{2^{k-1}}+ 5^{2^{k-1}}\\
    &= \left[2^{2^{k-1}} + (n-2)6^{2^{k-1}} +  5^{2^{k-1}} \right] + \left[ 3^{2^{k-1}} M + \sum_{k=1}^m \left( N_{k}^2 + 2 \right)^{2^{k-1}} - \left( M+m \right) 6^{2^{k-1}} \right]
    \end{split}
    \end{equation}\label{eq1}

Now we will show that for any $m \in \mathbb {N}$ 
$$
 3^{2^{k-1}} M + \sum_{k=1}^m \left( N_{k}^2 + 2 \right)^{2^{k-1}} - \left( M+m \right) 6^{2^{k-1}}  > 0
$$
by the following way -

\begin{align*}
 &3^{2^{k-1}} M + \sum_{i=1}^m \left( N_{i}^2 + 2 \right)^{2^{k-1}} - \left( M+m \right) 6^{2^{k-1}}\\
 &=  3^{2^{k-1}}\sum_{i=1}^m l_i  + \sum_{i=1}^m \left( (l_i + 2 )^2 + 2 \right)^{2^{k-1}} - \sum_{i=1}^m (l_i + 1 )6^{2^{k-1}}\\
 &= \sum_{i=1}^m l_i \left( 3^{2^{k-1}} -  6^{2^{k-1}} \right) - \sum_{i=1}^m  6^{2^{k-1}} + \sum_{i=1}^m \left( l_{i}^2 + 4 l_{i} + 6 \right)^{2^{k-1}}\\
 &\geq -3^{2^{k-1}}\sum_{i=1}^m l_i \left( 2^{2^{k-1}} - 1 \right) - \sum_{i=1}^m 6^{2^{k-1}} + \sum_{i=1}^m \left( l_{i}^4 + 8l_{i}^3 + 28 l_{i}^2 + 48 l_{i} +36 \right)^{2^{k-2}}\\
 &\geq  -3^{2^{k-1}}\sum_{i=1}^m l_i \left( 2^{2^{k-1}} - 1 \right) - \sum_{i=1}^m 6^{2^{k-1}} + \sum_{i=1}^m \left( 37 l_{i}^2 + 36 \right)^{2^{k-2}} \hspace{2mm}(\text{Using $l_{i}^4 \geq l_{i}^2$ and  $l_{i}^3 \geq l_{i}^2$  })\\
 &\geq  -3^{2^{k-1}}\sum_{i=1}^m l_i \left( 2^{2^{k-1}} - 1 \right) - \sum_{i=1}^m 6^{2^{k-1}} + \sum_{i=1}^m \left( 37 l_{i}^2 \right)^{2^{k-2}} + \sum_{i=1}^m \left( 36 \right)^{2^{k-2}} \\
 &\geq 3^{2^{k-1}}\sum_{i=1}^m l_i - \sum_{i=1}^m 6^{2^{k-1}}l_i - \sum_{i=1}^m 6^{2^{k-1}} + \sum_{i=1}^m (37l_{i}^2)^{2^{k-2}} + \sum_{i=1}^m 6^{2^{k-1}} \hspace{2mm} (\text{Using $ (37l_{i}^2)^{2^{k-2}} >  6^{2^{k-1}}l_i  $ })\\
 &>  3^{2^{k-1}}\sum_{i=1}^m l_i > 0
\end{align*}

Case I(b): Fix $a_{nn} = 1$ 

Since there are exactly $M+1$ number of rows with diagonal entries $1$ and first and $n$-th rows are among them therefore for these rows we have 
\begin{equation}\label{eq21}\
\begin{split}
&\left(\sum_{j=1}^n a_{1 j}^2 \right)^{2^{k-1}} + (M-1) \left(\sum_{j=1}^n a_{i j}^2 \right)^{2^{k-1}} + \left(\sum_{j=1}^n a_{n j}^2 \right)^{2^{k-1}}\\
&=  \left(\sum_{j=1}^n a_{1 1}^2+ a_{1 2}^2 \right)^{2^{k-1}} + (M - 1)  \left(\sum_{j=1}^n a_{i i}^2+ a_{i i+1}^2 + a_{i i-1}^2\right)^{2^{k-1}} + \left(\sum_{j=1}^n a_{n n}^2 + a_{n-1 n}^2 \right)^{2^{k-1}}\\
&= 2 (2^{2^{k-1}})  + (M - 1) 3^{2^{k-1}}
\end{split}
\end{equation}
here $i \in \{2,..n -1 \} $ such that $a_{ii} = 1$.

There are exactly $m$ number of rows with diagonal entries greater than $2$. For these rows we have
\begin{equation}\label{eq22}
\left(\sum_{i}\sum_{j=1}^n a_{i j}^2\right)^{2^{k-1}} = \left(\sum_{i}\sum_{j=1}^n a_{i i}^2+ a_{i i+1}^2 + a_{i i-1}^2\right)^{2^{k-1}} = \sum_{i=1}^m \left( N_{i}^2 + 2 \right)^{2^{k-1}} 
\end{equation}
here $i \in \{2,..n -1 \} $ such that $a_{ii} > 2$.

Also, there exists exactly $\{n - (m+M+1)\}$ number of rows with diagonal entries  $2$. For these rows we have
\begin{equation}\label{eq23}
\begin{split}
&\left( n - (m+M+1) \right) \left(\sum_{j=1}^n a{i j}^2 \right)^{2^{k-1}} \\
&= \left( n - (m+M+1) \right) \left(\sum_{j=1}^n a_{i i}^2+ a_{i i+1}^2 + a_{i i-1}^2\right)^{2^{k-1}} \\
&= \left( n - (m+M+1) \right) 6^{2^{k-1}}
\end{split}
\end{equation}
here $i \in \{2,..n -1 \} $ such that $a_{ii} = 2$.
Finally, by adding equations \ref{eq21},\ref{eq22},\ref{eq23} we get --
\begin{equation}
\begin{split}
   & Tr_{2^k}(A) \\
    &= 2( 2^{2^{k-1}})  + 3^{2^{k-1}} (M-1) + \sum_{i=1}^m \left( N_{i}^2 + 2 \right)^{2^{k-1}} + \left( n - (m+M+1) \right) 6^{2^{k-1}}\\
    &= \left[2^{2^{k-1}} + (n-2)6^{2^{k-1}} +  5^{2^{k-1}} \right] + \\
   & \left[ 3^{2^{k-1}} M + \sum_{i=1}^m \left( N_{i}^2 + 2 \right)^{2^{k-1}} - \left( M+m \right) 6^{2^{k-1}} +  2^{2^{k-1}} - 3^{2^{k-1}} + 6^{2^{k-1}} - 5^{2^{k-1}}   \right]
    \end{split}
    \end{equation}\label{eq2}
Then again for any $m \in \mathbb{N}$
\begin{align*}
    &\left[ 3^{2^{k-1}} M + \sum_{i=1}^m \left( N_{i}^2 + 2 \right)^{2^{k-1}} - \left( M+m \right) 6^{2^{k-1}} +  2^{2^{k-1}} - 3^{2^{k-1}} + 6^{2^{k-1}} - 5^{2^{k-1}}   \right]\\
    &= 3^{2^{k-1}} M + \sum_{i=1}^m \left( N_{i}^2 + 2 \right)^{2^{k-1}} - \left( M+m \right) 6^{2^{k-1}} + \left[(6^{2^{k-1}} - 5^{2^{k-1}}) -  (3^{2^{k-1}} - 2^{2^{k-1}})    \right] > 0\\
\end{align*}
Since both $3^{2^{k-1}} M + \sum_{i=1}^m \left( N_{i}^2 + 2 \right)^{2^{k-1}} - \left( M+m \right) 6^{2^{k-1}} > 0 $ and $\left[(6^{2^{k-1}} - 5^{2^{k-1}}) -  (3^{2^{k-1}} - 2^{2^{k-1}})\right] > 0$. 

Case I(c):
 Fix $a_{nn} = N_m$ 

Since there are exactly $M+1$ number of rows with diagonal entries $1$ and first row is one of them therefore for these rows we have 
\begin{equation}\label{eq31}\
\begin{split}
&\left(\sum_{j=1}^n a_{1 j}^2\right)^{2^{k-1}} + M \left(\sum_{j=1}^n a_{i j}^2 \right)^{2^{k-1}} \\
&=  \left(\sum_{j=1}^n a_{1 1}^2+ a_{1 2}^2 \right)^{2^{k-1}} + M  \left(\sum_{j=1}^n a_{i i}^2+ a_{i i+1}^2 + a_{i i-1}^2\right)^{2^{k-1}}
= 2^{2^{k-1}}  + M 3^{2^{k-1}}
\end{split}
\end{equation}
here $i \in \{2,..n -1 \} $ such that $a_{ii} = 1$.

There exists exactly $\{n - (m+M+1)\}$ number of rows with diagonal entries  $2$ . For these rows we have
\begin{equation}\label{eq32}
\begin{split}
&\left( n - (m+M+1) \right) \left(\sum_{j=1}^n a{i j}^2 \right)^{2^{k-1}} \\
&= \left( n - (m+M+1) \right) \left(\sum_{j=1}^n a_{i i}^2+ a_{i i+1}^2 + a_{i i-1}^2\right)^{2^{k-1}} \\
&= \left( n - (m+M+1) \right) 6^{2^{k-1}}
\end{split}
\end{equation}
here $i \in \{2,..n -1 \} $ such that $a_{ii} = 2$.

Also there are exactly $m$ number of rows with diagonal entries greater than $2$ and $n$-th row is one of them. For these rows we have
\begin{equation}\label{eq33}
\begin{split}
&\left(\sum_{i}\sum_{j=1}^n a_{i j}^2\right)^{2^{k-1}} + \left(\sum_{j=1}^n a_{n j}^2\right)^{2^{k-1}} \\
&= \left(\sum_{i}\sum_{j=1}^n a_{i i}^2+ a_{i i+1}^2 + a_{i i-1}^2\right)^{2^{k-1}}  + \left(\sum_{j=1}^n a_{n n}^2+ a_{n-1 n}^2 \right)^{2^{k-1}} \\ 
&= \sum_{i=1}^m-1 \left( N_{i}^2 + 2 \right)^{2^{k-1}} + \left( N_{m}^2 + 1 \right)^{2^{k-1}}
\end{split}
\end{equation}
here $i \in \{2,..n -1 \} $ such that $a_{ii} > 2$.

Finally, adding equations \ref{eq31},\ref{eq32},\ref{eq33} we get --
\begin{equation}
\begin{split}
   & Tr_{2^k}(A) \\
    &= 2^{2^{k-1}}  + M 3^{2^{k-1}} + \left( n - (m+M+1) \right) 6^{2^{k-1}} +  \sum_{i=1}^m-1 \left( N_{i}^2 + 2 \right)^{2^{k-1}} + \left( N_{m}^2 + 2 \right)^{2^{k-1}} \\
    &= \left[2^{2^{k-1}} + (n-2)6^{2^{k-1}} +  5^{2^{k-1}} \right] \\
   & \left[ 3^{2^{k-1}} M + \sum_{i=1}^{m-1} \left( N_{i}^2 + 2 \right)^{2^{k-1}} +\left( N_{m}^2 + 1 \right)^{2^{k-1}}  - \left( M+m \right) 6^{2^{k-1}}  + 6^{2^{k-1}} - 5^{2^{k-1}}   \right]
    \end{split}
    \end{equation}\label{eq3}
Then again for any $m \in \mathbb{N}$
\begin{align*}
    &\left[ 3^{2^{k-1}} M + \sum_{i=1}^{m-1} \left( N_{i}^2 + 2 \right)^{2^{k-1}} +\left( N_{m}^2 + 1 \right)^{2^{k-1}} - \left( M+m \right) 6^{2^{k-1}} + 6^{2^{k-1}} - 5^{2^{k-1}}   \right]\\
    &= 3^{2^{k-1}} \sum_{i=1}^{m-1} l_i + \sum_{i=1}^{m-1} \left( N_{i}^2 + 2 \right)^{2^{k-1}} - \left( \sum_{i=1}^{m-1} (l_i +1)  \right) 6^{2^{k-1}}+ \\
    &\left[(6^{2^{k-1}} - 5^{2^{k-1}}) + \left( N_{m}^2 + 1 \right)^{2^{k-1}} + 3^{2^{k-1}} l_m + (l_m +1)   6^{2^{k-1}}  \right] > 0
\end{align*}
From Case I we can say $3^{2^{k-1}} \sum_{i=1}^{m-1} l_i + \sum_{i=1}^{m-1} \left( N_{i}^2 + 2 \right)^{2^{k-1}} - \left( \sum_{i=1}^{m-1} (l_i +1)  \right) 6^{2^{k-1}}> 0 $ and 

\begin{align*}
    &\left[(6^{2^{k-1}} - 5^{2^{k-1}}) + \left( N_{m}^2 + 1 \right)^{2^{k-1}} + 3^{2^{k-1}} l_m - (l_m +1) 6^{2^{k-1}}  \right] \\
    &\geq \left[(6^{2^{k-1}} - 5^{2^{k-1}}) + \left( l_m^4 + 8l_m^3 + 26l_m^2 + 40 l_m +25 \right)^{2^{k-1}} + 3^{2^{k-1}} l_m - (l_m +1) 6^{2^{k-1}}  \right]\\
    &> \left[(6^{2^{k-1}} - 5^{2^{k-1}}) + (40 l_m)^{2^{k-2}} + 25^{2^{k-2}} + 3^{2^{k-1}} l_m - (l_m +1) 6^{2^{k-1}}  \right]\\
    &\geq 3^{2^{k-1}} l_m > 0
\end{align*}

This implies for any $m \in \mathbb{N}$ 
$$
\left[ 3^{2^{k-1}} M + \sum_{i=1}^{m-1} \left( N_{i}^2 + 2 \right)^{2^{k-1}} +\left( N_{m}^2 + 1 \right)^{2^{k-1}} - \left( M+m \right) 6^{2^{k-1}} + 6^{2^{k-1}} - 5^{2^{k-1}}   \right] > 0
$$
From the construction of $A$ we can say that for any matrix $B \in \mathcal{T}_n$ 
$$
Tr_{2^k}(B) \geq Tr_{2^k}(A) 
$$
On the other hand, we get from above three cases --
\begin{equation*}
\begin{split}
    Tr_{2^k}(A) &\geq min \Bigg\{ \left[2^{2^{k-1}} + (n-2)6^{2^{k-1}} +  5^{2^{k-1}} \right] + \left[ 3^{2^{k-1}} M + \sum_{k=1}^m \left( N_{k}^2 + 2 \right)^{2^{k-1}} - \left( M+m \right) 6^{2^{k-1}} \right],\\
    &\left[2^{2^{k-1}} + (n-2)6^{2^{k-1}} +  5^{2^{k-1}} \right] + \\
    &\left[3^{2^{k-1}} M + \sum_{i=1}^m \left( N_{i}^2 + 2 \right)^{2^{k-1}} - \left( M+m \right) 6^{2^{k-1}} +  2^{2^{k-1}} - 3^{2^{k-1}} + 6^{2^{k-1}} - 5^{2^{k-1}}\right],\\
    & \left[2^{2^{k-1}} + (n-2)6^{2^{k-1}} +  5^{2^{k-1}} \right] + \\
    &\left[ 3^{2^{k-1}} M + \sum_{i=1}^{m-1} \left( N_{i}^2 + 2 \right)^{2^{k-1}} +\left( N_{m}^2 + 1 \right)^{2^{k-1}}  - \left( M+m \right) 6^{2^{k-1}}  + 6^{2^{k-1}} - 5^{2^{k-1}}  \right]    \Bigg\} \\
    &\geq \left[2^{2^{k-1}} + (n-2)6^{2^{k-1}} +  5^{2^{k-1}} \right]
    \end{split}
\end{equation*}
And among all these cases only in the first case we can Take $m=0$ which implies $Tr_{2^k}(A) \geq \left[2^{2^{k-1}} + (n-2)6^{2^{k-1}} +  5^{2^{k-1}} \right]$ . And in the other two cases $Tr_{2^k}(A)$ is strictly greater than $\left[2^{2^{k-1}} + (n-2)6^{2^{k-1}} +  5^{2^{k-1}} \right]$. From here we can conclude that --
\begin{equation}\label{eqn3}
    Tr_{2^k}(B) \geq \left[2^{2^{k-1}} + (n-2)6^{2^{k-1}} +  5^{2^{k-1}} \right]
\end{equation}
for all $B \in \mathcal{T}_n$ 

Case II: Now we prove that the inequality on the lower bound of trace $2^k$-measure is also true for matrices in $\mathcal{S}_n$. 

By applying symmetrization map $\phi$ from definition \ref{d25} we get that every matrix $A \in \mathcal{S}_n$ is similar to some matrix $\phi(A) \in \mathcal{T}_n$. Therefore the multiset of eigenvalues of $A$ and $\phi(A)$ are the same. Hence from equation \ref{eqn3} we get
$$
Tr_{2^k}(A) = Tr_{2^k}(\phi(A)) \geq  2^{2^{k-1}} + (n-2)6^{2^{k-1}} +  5^{2^{k-1}}
$$
for all $A \in \mathcal{S}_n$.

\textit{Second Part:} Putting $k=1$ in equation \eqref{eqn3}  we get --
\begin{equation}\label{eq411}
   Tr_2(A) \geq  6n - 5 
\end{equation}

For all $A \in \mathcal{T}_n \cup \mathcal{S}_n$.
\end{proof}

\begin{cor}\label{11}
Let $U_k$ and $V_k$ be the set of absolute trace-$2^k$ measure of matrices in $\mathcal{S}$ and $\mathcal{T}$ respectively. Then the following are true --

\setlist[enumerate,1]{label={(\roman*)}}\setlist[enumerate,1]{label={(\roman*)}}
\begin{enumerate}
    \item $U_k = V_k$ $\forall k \in \mathbb{N} \cup \{0\}$. \label{c1}
    \item the smallest limit point of $U_1$ is $6$. \label{c2}
    \item the smallest limit point of $U_k$ can not be less than $5^{2^{k-1}}$ for any positive integer $k > 1$. \label{c3}
\end{enumerate}
\end{cor}

\begin{proof}: 

\textit{Part \ref{c1}:} Let $u \in U_k \implies u = \overline{Tr_{2^k}(A)}$ for some $A \in \mathcal{S}$. Applying the symmetrization map $\phi$ on $A$ we get $\phi(A) \in \mathcal{T}$. Since $\phi$ preserves the multi-set of eigenvalues of $A$ and the order of $A$, therefore -- 
$$
\overline{Tr_{2^k}(A)} = \overline{Tr_{2^k}(\phi(A))} = u
$$
This implies, $u \in V_k$.

Conversely, let $v \in V_k \implies v = \overline{Tr_{2^k}(B)}$ for some $B \in \mathcal{T}$. Since from (Proposition $1$, \cite{SM}) we know that $\phi : \mathcal{S} \rightarrow \mathcal{T}$ is surjective therefore $B = \phi(A)$ for some  $A \in \mathcal{S}$. Then again, by the same argument we get --
$$
\overline{Tr_{2^k}(B)} = \overline{Tr_{2^k}(\phi(A))} =\overline{Tr_{2^k}(A)} = v
$$
This implies, $ v \in U_k $.

\textit{Part \ref{c2}:} Consider the following $n \times n$ matrix 

$$ M_n = 
\begin{pmatrix} 1 & -1 & 0 & \cdots & 0\\
-1 & 2 & -1 & \cdots & 0 \\
0 & -1 & 2 & \cdots & 0 \\
\vdots  & \vdots  & \vdots & \ddots & \vdots  \\
0 & 0 & 0 & \cdots & 2
\end{pmatrix} 
$$

This matrix is 

\begin{itemize}

     \item \textit{symmetric, integer:} $M_n = (m_{ij})$ where $m_{nn} = 1$, $m_{ii} = 2$ for $i < n$ and $m_{i, i+1} = m_{i+1 , i} = -1$ and all other entries are $0$. Therefore by definition it follows.
     
    \item \textit{positive definite:} Since $M_n$ is a symmetric diagonally dominant (not strictly) therefore by Gershgorin's circle theorem \cite{G} eigenvalues of $M$ are all non-negative. Again, $M$ is a non-singular matrix. This implies every eigenvalue of $M$ must be positive. 
    
   \item \textit{connected:} By definition of connectedness given in section \ref{S},  $i$-th point of the lattice $G_{M_n}$ associated with the matrix $M_n$ is connected with $i+1$-th point for each $i = 1,...,n$. It means, this is a \textit{path}. 
   
\end{itemize}

   For each $n \in \mathbb{N}$ let us consider the sequence ${M_n} \in \mathcal{S} \cap \mathcal{T}$ of this form. then from the proof of Theorem \ref{1} we get --

$$
Tr_2(M_n) = \sum_{i=1}^n \lambda_i^{2^k} = \sum_{i=1}^n \left(\sum_{j=1}^n a_{ij}^2\right) = 6n - 5
$$
Then limit of $\overline{tr_2(M_n)} = \frac{6n - 5}{n}$ of this sequence is $6$.
   
   Now we show that $6$ is the smallest limit point of $U_k$ for any $k \in \mathbb{N}$. 
   
   Let us assume the contrary.
   
   Take any $\epsilon > 0$. Let $M^{'} = (m_{ij})\in \mathcal{S}\cap \mathcal{T}$ be a matrix with  $\overline{Tr_2 }(M) < 6 - \epsilon $ . But by Theorem \ref{1} 
   $$
   \overline{Tr_2 (M^{'})} \geq \frac{6n - 5}{n}
   \implies n < \frac{5}{\epsilon}
   $$
Let $n=n_0$ is fixed and let $\{ \mu_1,..., \mu_{n_0}\}$ be the multi-set of eigenvalues of $M'_n$. Also let $\mu_1,..., \mu_{r_0} \leq 1$ for some non-negative integer $r_0 < n_0$. Then, 
$$
\left(\overline{\sum_{i=1}^{n_{0}} \mu_i}\right) = \frac{\sum_{i=1}^{r_0} \mu_i + \sum_{i= r_0 + 1}^{n_0} \mu_i}{n_0} \leq \frac{r_0 + \sum_{i= r_0 + 1}^{n_0} \mu_i^2}{n_0} \leq \frac{r_0 + (6n_0 - 5 )}{n_0} < \frac{7n_0 - 5}{n_0}
$$
This bound on absolute trace gives a finite set of possibilities for diagonal entries of $M'$.
Also, if we consider any principal sub-matrix 
$$
\begin{pmatrix}
m_{ii} & m_{ij} \\
m_{ij} & m_{jj}
\end{pmatrix}
$$

of $M'$

Since $M'$ is a positive-definite matrix each of its principal sub-matrix must be the same. This implies, 
$$
|m_{ij}| < \sqrt{m_{ii}m_{jj}}
$$

It means, absolute value of each $m^n_{ij}$ is bounded by some function of diagonal entries(which is obviously bounded). Thus for each of the entries of such matrices we have only finitely many choices. Hence the number of choices of such matrices is finite and this leads to a contradiction.

\textit{Part \ref{c3}} We first show that, the smallest limit point of $U_{k}$ can not be less than $5^{2^{k-1}}$.

Let  $A\in \mathcal{S}\cap \mathcal{T}$ with $\overline{Tr_{2^k} (A)} < 5^{2^{k-1}} - \epsilon$. Then by Theorem \ref{1}
$$
\overline{Tr_{2^k} (A)} \geq   2^{2^{k-1}} + (n-2)6^{2^{k-1}} +  5^{2^{k-1}}
$$
 where  $k \in \mathbb{N}$.

This implies, 
\begin{eqnarray*}
   5^r - \epsilon=\overline{Tr_{2r} (A)} &\geq&  \frac{2^r + (n-2)6^r +  5^r}{n}
   \implies n^2 5^r > (n-1)5^r - 2^r - (n-2)6^r \geq n \epsilon
\end{eqnarray*}

where $r = 2^{k-1}$ .

Using this inequality we get $n < \frac{5^r}{\epsilon}$.This again gives a finite upper bound of absolute trace of matrix $A$ which leads to a finite set of possibilities of eigenvalues of $A$. Rest of the proof is exactly similar to the proof of part \ref{c2}.
\end{proof}

\begin{cor}\label{12}
Let 
$$
p(x) = x^n - a_1x^{n-1} + ... + (-1)^n a_0 \in \mathbb{Z}[X]
$$

be a monic, irreducible totally positive polynomial with trace $2n - 1$. Then $p(x)$ is not a characteristic polynomial of some matrix in $\mathcal{S}_n \cup \mathcal{T}_n $ if $a_2 > 4n^2 - 10n + 6 $.

In general, if we know $a_1$,..., $a_{2^k-1}$ we obtain an integer $r_k$ such that $p(x)$ is not a characteristic polynomial of a matrix in $\mathcal{S}_n \cup \mathcal{T}_n $ if $a_{2^k}  > \frac{r_k}{2^k} \hspace{5mm} \forall k \in \mathbb{N}$.

\end{cor}

\begin{proof}:
Let $A_p$ be a matrix in $\mathcal{S}_n \cup \mathcal{T}_n$ with characteristic polynomial $p(x)$.

\textit{Part I:}  Let us assume the contrary. Let $\{ \lambda_1, \lambda_2,..., \lambda_n \}$ = multiset of eigenvalues of $A_p$. 

Then $a_1 = \sum_{i=1}^n \lambda_i = 2n - 1$ and $a_2 = \sum_{1 \leq i < j \leq n }^n \lambda_i \lambda_j  > 4n^2 - 10n + 6 $  implies 
$$
Tr_2(A_p) = \sum_{i=1}^n (\lambda_i)^2 =\left(\sum_{i=1}^n \lambda_i \right)^2  -   \sum_{1 \leq i < j \leq n }^n \lambda_i \lambda_j < (2n - 1)^2 - ( 4n^2 - 10n + 6 )  =  6n - 5 
$$
But from Theorem \ref{1} we get, for any $A \in \mathcal{S}_n \cup \mathcal{T}_n$ then $Tr_2(A) \geq 6n - 5 $. This leads to a contradiction.

\textit{Part II:} 
 
 Let us denote $\sum_{i=1}^n \alpha_i^l = s_l$
 
 From Newton's Identity for symmetric polynomial we know that,
 
 $$
 s_1 = a_1 
 $$
 And,
 $$
 s_l = (-1)^{l-1}l a_{l} + \sum_{i=1}^{l-1} (-1)^{(l-1)+i}a_{l-i}s_i  \hspace{5mm} \text{ for } \; 1\leq l \leq 2^k-1
 $$
 
 Therefore, if we know $a_1$, $a_2$,..., $a_{2^k-1}$ then we get $s_1$, $s_2$,..., $s_{2^k-1}$.
 
 Again, from Newton's identity we get --
 
\begin{equation}\label{e1}
a_{2^k} = \frac{\sum_{i=1}^{2^k-1} (-1)^{i-1} a_{2^k-i-1}s_i + (-1)^{2^k - 1}s_{2^k}}{2^k} = \frac{\sum_{i=1}^{2^k-1} (-1)^{i-1} a_{2^k-i-1}s_i - s_{2^k}}{2^k}
\end{equation}

Let us denote $\sum_{i=1}^{2^k-1} (-1)^{i-1} a_{2^k-i-1}s_i = f_k$.

From Theorem \ref{1} we get that if $s_{2^k}$ is less than the bound, say $e_k$, given in \eqref{eq411} then $A_p$ cannot be a matrix in $\mathcal{S}_n \cup \mathcal{T}_n$. Using this in equation \ref{e1} we get that,
$$
a_{2^k} > \frac{f_k - e_k}{2^k}
$$
Taking $f_k - e_k = r_k$ we get the result.

\end{proof}

From this result we can say when an irreducible polynomial can not be characteristic polynomial of a symmetric integer connected positive definite matrix.

\section{Statements and Proofs of Theorem \ref{2} And Theorem \ref{3}}

In our next result, we have tried to extend it for minimal polynomial of a symmetric integer positive semidefinite  matrix. The matrix doesn't need to be connected.

In \cite{M}, as a corollary of Theorem 2, it is proved the following --

 \textit{An $n\times n $ complex positive semidefinite matrix of rank $r$ whose graph has $s$ connected components, whose diagonal entries are integers, and whose non-zero off-diagonal entries have modulus at least one, has trace at least} $n+r-s$.
 
 Using this result, we will prove the following Theorem.

 \begin{thm}\label{2}

Let $p_i(x)$ be monic, irreducible, totally positive polynomials for $i= 1,...,s$. If $tr(p_i) < 2 deg(p_i)$ $ \forall i$ then $m(x) = xp_1(x)...p_s(x)$ can not be written as minimal polynomial of an integer symmetric matrix $A$ of the following type:

Every connected component of $A$ is positive semi-definite but no component is positive definite.

\end{thm}

\begin{proof}: Let us assume the contrary. Let $m(x)$ be the minimal polynomial of an integer symmetric matrix $A_m$ with $s'$ number of connected components of the given form. Then the characteristic polynomial $\chi_m(x)$ of $A_m$ is of the form
 $$
 \chi_m(x) = x^lp_1^{n_1}(x)p_2^{n_2}(x)...p_s^{n_s}(x)
 $$
 for some $n_1, n_2,..., n_s \in \mathbb{N}$.
 
Let  $ deg(p_i) = d_i$
 
 If $Tr(p_i(x)) < 2 d_i$ then $Tr(\chi_m(x)) = Tr(A_m) < \sum_{i=1}^s 2n_id_i$

 Now, $deg(\chi_m(x)) = ord(A_m) = \sum_{i=1}^s n_id_i + l = n $ (say)  

and  $rank(A_m) = \sum_{i=1}^s n_id_i = r$ (say)
 
 Here, $deg(f(x)) =$ degree of $f(x)$ for any polynomial $f(x)$ 

and $ord(A) = $ order of matrix $A$.
 
 By given condition, every connected component of $A_m$ is positive semi-definite. This implies, each component has nullity at least $1$ $\implies$ nullity of $A_m$ must be at least $s'$. That means, $l \geq s'$
 
 From here, we can say that, $tr(\chi_m(x)) < \sum_{i=1}^s 2n_id_i \leq 2\sum_{i=1}^s n_id_i + l - s' =(\sum_{i=1}^s n_id_i + l ) + \sum_{i=1}^s n_id_i - s' = n + r - s'$
 
 But this is not possible by above mentioned result in \cite{M}. Hence proved.
 
\end{proof}
In general, from this theorem we get that, if $A$ is a matrix such that nullity of $A \geq$ rank of $A$ then any polynomial $m(x)$ of the given form can not be minimal polynomial of $A$. 

Using this result we can produce a large number of counter examples of Estes-Guralnick conjecture \cite{EG}.

 From Smyth's famous result \cite{S1} on ``Schur-Siegel-Smyth" trace problem we know that the set of absolute traces of totally positive algebraic integers is dense in $[2, \infty)$. This means that for any arbitrary but fixed real number $r \geq 2$ there exists a sequence $({r_n})_{n\in \mathbb{N}}$ converges to $r$. Here, for all $n \in \mathbb{N}$,  $r_n =$ absolute trace of $\alpha_n$ for some totally positive algebraic integer $\alpha_n$. Moreover, from the same paper we also know that the set of \textit{absolute trace-2 measure} of totally positive algebraic integers is dense in $[6, \infty)$. Here we prove a matrix analogue of these two results.

For proving the part \ref{t1} of Theorem \ref{3} we will use the following theorem of Mckee and Yatsyna \cite{MY} --

\textit{Let $A$ be an $n \times n$, connected, integer, symmetric matrix. If all  the eigenvalues of $A$ are strictly positive, then $Tr(A) \geq 2n - 1$}.

And for proving the part \ref{t2} of the same Theorem we will use the Theorem \ref{1}. In particular, we will use that, for any  $n \times n$, connected, integer, symmetric matrix $A$, $Tr_2(A) \geq 6n - 5$.

\begin{thm}\label{3}
 Let $\mathcal{C}$ be the set of all integer, symmetric, connected, positive definite matrices. Then

 \setlist[enumerate,1]{label={(\roman*)}}\setlist[enumerate,1]{label={(\roman*)}}

 \begin{enumerate}
     \item $X \coloneqq \{ \overline{Tr(A)} | A \in \mathcal{C} \}$ is dense in $[2, \infty) $. \label{t1}

     \item $Y \coloneqq \{ \overline{Tr_2(A)} | A \in \mathcal{C} \}$ is dense in $[6, \infty) $. \label{t2}
 \end{enumerate}

\end{thm}

\begin{proof}: 

\textit{Part\ref{t1}}: In this part we will show that the trace measure of the set of all integer, symmetric, connected, positive definite matrices is dense in $[2, \infty) $.

From the result of Mckee and Yatsyna \cite{MY} we get that $2$ is the smallest limit point of $X$. Therefore it is enough to prove the result for any real number $r > 2$.

Let $\{\alpha_n\}_{n\in \mathbb N}$ be a sequence of totally positive algebraic integers corresponding to $r$ and let $N_n$ be the number of conjugates of $\alpha_n$  for all $n\in \mathbb{N}$.

Also let 
$$
P_n = x^{N_n} - a_{N_n}x^{N_n-1} + ... + (-1)^{N_n} a_1 \in \mathbb{Z}[X]
$$

be the minimal polynomial of $\alpha_n$ for all $n\in \mathbb{N}$ .

Then the absolute trace of $\alpha_n \coloneqq r_n= \frac{a_{N_n}}{N_n}$. 

By assumption, for any $r > 2$ there exists $M\in \mathbb{N}$ such that for all $n \geq M$  $\frac{a_{N_n}}{N_n} \geq 2$. 

Now for all $n \geq M$ we construct a sequence of matrices \{$T_{n}\}_{n\in \mathbb{N}} $ of order $N_n \times N_n$ by the following way--

$$ T_{n} = 
\begin{pmatrix} a_{N_n}- 2(N_n -1) & -1 & 0 & \cdots & 0\\
-1 & 2 & -1 & \cdots & 0 \\
0 & -1 & 2 & \cdots & 0 \\
\vdots  & \vdots  & \vdots & \ddots & \vdots  \\
0 & 0 & 0 & \cdots & 2
\end{pmatrix} 
$$

Since 
$$
\frac{a_{N_n}}{N_n} \geq 2
\implies a_{N_n} - 2(N_n -1) \geq 1
$$. 

Since $T_{n}$ is a symmetric, diagonally dominant and non-singular matrix, it must be positive definite. Also, by construction it is a connected and integer matrix. This means, $T_n \in \mathcal{C}$ for all $n\in \mathbb{N}$. Now,
$$
\text{Absolute trace of} \ T_n = 
\frac{\text{sum of diagonal entries of} \ T_n}{\text{order of} \ T_n} = \frac{a_{N_n}}{N_n} = \text{absolute trace of} \ {\alpha_n}
$$

Hence, it converges to $r$ for any given $r > 2$. 

\textit{Part \ref{t2}:} In this part we will show that the trace-$2$ measure of the set of all integer, symmetric, connected, positive definite matrices is dense in $[6, \infty) $.

In part \ref{c2} of Corollary \ref{11} we have already shown that $6$ is the smallest limit point of $Y$. Therefore it is enough to prove this result for any real number $r>6$.

Let $\{\alpha_n\}_{n\in \mathbb{N}}$ be a sequence of totally positive integers corresponding to $r$ such that the degree of $\alpha_n$ is $N_n$ and the minimal polynomial of $\alpha_n$ is 
$$
P_n = x^{N_n} - a_{N_n}x^{N_n-1} + ... + (-1)^{N_n} a_1 \in \mathbb{Z}[X]
$$
Then $r_n \coloneqq \frac{Tr_2(\alpha_n)}{N_n} = \frac{a_{N_n}^2 - 2a_{N_n -1}}{N_n}$. Since $r>6$ there exists $M\in \mathbb{N}$ such that for all $n \geq M$,  $r_n >6$.

Now let us consider a sequence of matrices \{$L_{n}\}_{n\in \mathbb{N}} $ of order $N_n \times N_n$ given by --

$$ L_{n} = \{l^{(n)}_{ij}\}_{i \leq i,j \leq n} =
\begin{pmatrix} w_1 & -1 & 0 & 0 & 0 &\cdots & 0\\
-1 & w_2 & -1 & 0 & 0 & \cdots & 0 \\
0 & -1 & w_3 & -1 & 0 & \cdots & 0 \\
0 & 0 & -1 & w_4 & 0 & \cdots & 0 \\
0 & 0 & 0 & -1 & 2 & \cdots & 0 \\
\vdots  & \vdots  & \vdots & \vdots & \vdots & \ddots & \vdots  \\
0 & 0 & 0 & 0 & 0 & \cdots & 2
\end{pmatrix} 
$$
From the construction it is clear that we are considering $N_n \geq 5$. 
Now for $n < M$ we are simply putting $w_1 = 1$ and $w_2 = w_3 = w_4 = 2$. For $n \geq M$ we are choosing $w_1, w_2, w_3, w_4$ such that $0 \leq w_1 \leq w_2\leq w_3 \leq w_4 $ by the following way --

Since for any $n \geq M$,  $r_n > 6$ therefore 
$$
a_{N_n}^2 - 2a_{N_n -1} - 2(N_n -1) - 4(N_n - 4)
$$
is a positive integer greater than $18$.

From Lagrange's four square theorem we know that every positive integer can be represented as sum of four integer squares. therefore we can write 
\begin{equation}\label{eq111}
    a_{N_n}^2 - 2a_{N_n -1} - 2(N_n -1) - 4(N_n - 4) = b_1^2 + b_2^2 + b_3^2 + b_4^2  
\end{equation}
 such that $0 \leq b_1 \leq b_2\leq b_3 \leq b_4 $.

 Then $b_4$ is always greater than $1$ (because if all $b_i$'s are $0$ or $1$ then $ \sum_{i=1}^4 b_i^2 \leq 4$).

 We take $w_4 = b_4$. 

 If $b_3 \geq 2$ we take $w_3 = b_3$. If $b_3 = 0$ or $1$ we take $w_3 = 2$ (that means, $w_3 = b_3 +0$ or $w_3 = b_3 +1$ or $w_3 = b_3 +2$). We choose $w_2$ by the same way.
 
 If $b_1 \geq 1$ we choose $w_1 = b_1$. If $b_1 = 0$ we choose $w_1 = 1$.

Then by construction $L_n \in \mathcal{C}$ for all $n \in \mathbb{N}$ and 
\begin{equation}\label{eq222}
\sum_{i=1}^4 w_i^2 =  \sum_{i=1}^4 b_i^2 + \kappa_{N_n}
\end{equation}
where $0\leq \kappa_{N_n} \leq 9$ is an integer.

Then from equation \ref{eq111} and equation \ref{eq222} we get 
$$
Tr_2 (L_n) = \sum (l^{(n)}_{ij})^2 = \sum_{i=1}^4 w_i^2 + 2(N_n -1) + 4(N_n - 4) = a_{N_n}^2 - 2a_{N_n -1} + \kappa_{N_n}
$$
This implies 
$$
\overline{Tr_2 (L_n)} = \frac{a_{N_n}^2 - 2a_{N_n -1} + \kappa_{N_n}}{N_n}
$$
Since $\kappa_{N_n}$ is finite and $r_n = \frac{a_{N_n}^2 - 2a_{N_n -1}}{N_n}$ converges to $r$ then $\overline{Tr_2 (L_n)}$ also converges to $r$.

Hence completed.




\end{proof}

\end{document}